\documentclass[psamsfonts]{amsart}

\usepackage{amssymb,amsfonts}
\usepackage[all,arc]{xy}
\usepackage{enumerate}
\usepackage{mathrsfs}
\usepackage{graphicx}

\usepackage{orcidlink}

\newtheorem{theorem}{Theorem}[section]
\newtheorem{corollary}[theorem]{Corollary}
\newtheorem{proposition}[theorem]{Proposition}

\theoremstyle{definition}
\newtheorem{definition}[theorem]{Definition}

\newtheorem{example}[theorem]{Example}
\newtheorem{examples}[theorem]{Examples}

\newtheorem{remark}[theorem]{Remark}

\DeclareMathOperator{\Id}{Id}

\DeclareMathOperator{\Spec}{Spec}

\DeclareMathOperator{\diam}{diam}

\DeclareMathOperator{\IG}{IG}


\numberwithin{equation}{section}

\begin{document}
	
	\title{Algebraic properties of Indigenous semirings}
	
	\author[H. Behzadipour]{Hussein Behzadipour\orcidlink{0000-0001-7037-4110}}
	
	\address{Hussein Behzadipour\\
		Department of Mathematics \\
		Sharif University of Technology\\
		Tehran\\
		Iran}
	\email{hussein.behzadipour@gmail.com}
	
	\author[H. Koppelaar]{Henk Koppelaar\orcidlink{0000-0001-7487-6564}}
	
	\address{Henk Koppelaar\\
		Faculty of Electrical Engineering, Mathematics and Computer Science\\
		Delft University of Technology\\ Delft\\ The Netherlands}
	
	\email{koppelaar.henk@gmail.com}
	
	\author[P. Nasehpour]{Peyman Nasehpour\orcidlink{0000-0001-6625-364X}}
	
	\address{Peyman Nasehpour\\
		Education Department\\ The New York Academy of Sciences \\ New York, NY, USA}
	
	\email{nasehpour@gmail.com}
	
	\subjclass[2010]{16Y60, 13A15, 01A07.}
	
	\keywords{Indigenous semirings, Information algebras, Graph invariants}
	
	\begin{abstract}
	In this paper, we introduce Indigenous semirings and show that they are examples of information algebras. We also attribute a graph to them and discuss their diameters, girths, and clique numbers. On the other hand, we prove that the Zariski topology of any Indigenous semiring is the Sierpi\'{n}ski space. Next, we investigate their algebraic properties (including ideal theory). In the last section, we characterize units and idempotent elements of formal power series over Indigenous semirings.
	\end{abstract}

	\maketitle
	
	\section{Introduction}
	
	Inspired by the Indigenous number systems (cf. \cite{BenderBeller2021}, \cite{Everett2005}, \cite{Gordon2004}, \S1 in \cite{Ifrah2000}, and \cite{Vandendriessche2022}) and the Indigenous presemiring $\mathcal{M} = \{1,2,3,m\}$ proposed by Sibley in Example 5 on p. 419 of his educational book on abstract algebra \cite{Sibley2021}, we give a general definition for Indigenous semirings and investigate their algebraic properties. It turns out that the Indigenous semirings belong to a particular family of semirings called information algebras which have applications in various fields of science and engineering and have attracted the interests of some authors since 1972 (see Remark \ref{Informationalgebrasrem}).
	
	 From an algebraic perspective, one could argue that a number is not its physical appearance or digit representation. A number is an abstract mathematical object, whereas its appearance is a sequence of symbols on a paper (or a sequence of bits in computer memory, or a sequence of sounds if read aloud). We never see a number itself, but always its representation. So, we become accustomed to identifying a number by its representation. In Ethnomathematics (one of the subjects of this paper) it is similar. This confusing identification exposes one of the difficulties of the field. Other difficulties are exemplified in Examples \ref{historicalexamplesIndigenous} of this paper. A plea to alleviate the study problems of indigenous number systems is to develop cultural tools for numerical cognition \cite{BenderBeller2018}. For instance, the Yapese indigenous money counting is based on stone disks called ``rai stones''. A typical ``rai stone'' is carved out of crystalline limestone and shaped like a disk with a hole in the center. The smallest may be 3.5 centimeters in diameter while the largest extant stone is 3.6 meters in diameter and 50 centimeters thick, and weighs 4,000 kilograms \cite{Gillilland1975}. 
	 
	 This paper develops the underlying algebra (presemiring) of indigenous number systems to an unprecedented degree. Since the language of semiring-like algebraic structures is not standardized yet, we first introduce some terminology. Recall that a bimagma $(R,+,\cdot)$ is a ringoid \cite[p. 206]{Rosenfeld1968} if the binary operation ``$\cdot$'' (multiplication) distributes on the binary operation ``$+$'' (addition) from both sides. A ringoid $(R,+,\cdot)$ is a presemiring if $(R,+)$ is a commutative semigroup and $(R,\cdot)$ a semigroup \cite[Definition 4.2.1]{GondranMinoux2008}. A presemiring is commutative if $(R,\cdot)$ is a commutative semigroup. In this paper, all presemirings are supposed to be commutative. A presemiring $S$ is a semiring if it has a neutral element $0$ for its addition which is also an absorbing element for its multiplication and it has also a neutral element $1 \neq 0$ for its multiplication \cite[p. 1]{Golan1999(b)}. A semiring $S$ is entire if it is zero-divisor free, i.e., $ab = 0$ implies $a = 0$ or $b = 0$, for all $a$ and $b$ in $S$. A semiring $S$ is zerosumfree if $a+b = 0$ implies $a = b = 0$, for all $a$ and $b$ in $S$. A semiring $S$ is an information algebra if it is both zero-divisor free and zerosumfree \cite[p. 4]{Golan1999(b)}. A nonempty set $I$ of a semiring $S$ is an ideal of $S$ if $(I,+)$ is a submonoid of $(S,+)$ and $SI \subseteq I$ \cite{Bourne1951}. We collect all ideals of a semiring $S$ in $\Id(S)$. An ideal $M$ of a semiring $S$ is maximal if there is no ideal properly between $M$ and $S$. A semiring $S$ is local if it has a unique maximal ideal. An ideal $P$ of a semiring $S$ is prime if $P \neq S$ and $IJ \subseteq S$ implies $I \subseteq S$ or $J \subseteq S$, for all ideals $I$ and $J$ of $S$. All prime ideals of a semiring $S$ are collected in $\Spec(S)$. In a (commutative) semiring $S$, an ideal $P \neq S$ is prime if and only if $ab \in P$ implies either $a \in P$ or $b \in P$, for all elements $a$ and $b$ in $S$ \cite[Corollary 7.6]{Golan1999(b)}. 
	
	The first section of the paper is devoted to some results for entire semirings. Recall that for each ideal $I$ of a semiring $S$ \[V(I) = \{P \in \Spec(S) : P \supseteq I\}.\] It is, then, easy to verify that $\mathcal{C} = \{V(I) : I \in \Id(S)\}$ is the family of closed sets for a topology on $X = \Spec(S)$, called the Zariski topology \cite[p. 89]{Golan1999(b)}. A topological space with exactly two points and three closed subsets is called the Sierpi\'{n}ski space (see \cite[p. 17]{ArensDugundji1951} and \cite[Exercise 1.7]{Rotman1988}). In Theorem \ref{Zariskitopologyentiresemiringwithtwoprimes}, we show that the Zariski topology of an entire semiring with exactly one nonzero prime ideal is the Sierpi\'{n}ski space. Then, we proceed to show that the localization of an information algebra is an information algebra (see Proposition \ref{Localizationofinformationalgebras}).
	
	Let us recall that if $S$ is a semiring, then $\Id(S)$ equipped with addition and multiplication of ideals is a semiring \cite[Proposition 6.29]{Golan1999(b)}. In Theorem \ref{Semiringidealsinformationalgebra}, we prove that $\Id(S)$ is an information algebra if and only if $S$ is an entire semiring.
	
    Inspired by Example 5 on p. 419 in \cite{Sibley2021}, in \S\ref{sec:indigenousgraphs} of our paper, we define Indigenous addition and multiplication on $I_k = \mathbb{N}_k \cup \{m\}$, where $\mathbb{N}_k$ is the set of positive integer numbers less than $k+1$ and $m$ is just a symbol standing for ``many'' (see Definition \ref{Indigenouspresemiringdef}), and next in Proposition \ref{Indigenoussemiringpro}, we show that the bimagma $(I_k, \oplus, \odot)$ is a unital presemiring. For some historical examples of Indigenous presemirings, check Examples \ref{historicalexamplesIndigenous}.

    In Definition \ref{Indigenousgraphs}, we attribute a graph $\IG_k$ to any Indigenous presemiring $I_k$ where its vertices' set is $I_k$ and $\{a,b\}$ is an edge of $\IG_k$ if $a \neq b$ are elements of $I_k$ with $ab = m$. In the rest of \S\ref{sec:indigenousgraphs}, we discuss some graph invariants of Indigenous graphs. For example in Theorem \ref{diamIG}, we prove that the Indigenous graph $\IG_k$ is a connected graph with $\diam(\IG_1) = 1$ and $\diam(\IG_k) = 2$ if $k > 1$. In Theorem \ref{girthIG}, we show that the girth of the Indigenous graph $\IG_k$ is $3$ if $k \geq 3$, and is infinity, otherwise. Finally in Theorem \ref{cliqueIG}, we prove that the clique number of the Indigenous graph $\IG_k$ is at least $\left\lfloor \frac{k}{2} \right\rfloor + 1$, for any positive integer $k$.

    By annexing $0$ to the Indigenous presemiring $I_k$, we define the Indigenous semiring $S_k$ and in Theorem \ref{GeneralpropertiesofIndigenoussemirings}, we discuss algebraic properties of the Indigenous semirings. In this theorem with 12 items, we show that, for instance, any Indigenous semiring $S_k$ is a local information algebra and $\{0\}$ and $S_k \setminus \{1\}$ are the only prime ideals. A corollary to this statement is that the Zariski topology of any Indigenous semiring $S_k$ is the Sierpi\'{n}ski space. Next, we verify that $S_k$ is not a semidomain (recall that a semiring $S$ is a semidomain \cite{Nasehpour2019} if $S\setminus\{0\}$ is a multiplicatively cancellative monoid). In the same result, we show that any Indigenous semiring is austere and discuss the algebra of the ideals of the Indigenous semirings. In fact, we show that $(\Id(S_k)\setminus\{\mathbf{0}\}, \cdot)$ is a nilpotent monoid with the absorbing element $\mathfrak{s}_k = \{0,m\}$, where by $\mathbf{0}$, we mean the zero ideal $\{0\}$ of $S_k$.

    In \S\ref{sec:distinguishedelements}, we characterize units and idempotent elements of polynomials and formal power series over the Indigenous semirings (see Proposition \ref{UnitsIndigenoussemirings}, Proposition \ref{IdempotentsIndigenoussemirings1}, and Theorem \ref{IdempotentsIndigenoussemirings2}).
    
    Golan's book \cite{Golan1999(b)} is a general reference for semiring theory, and our terminology closely follows it.
  	
\section{Some results in entire semirings}\label{sec:entiresemirings}
  	
\begin{theorem}\label{Zariskitopologyentiresemiringwithtwoprimes}
The Zariski topology of an entire semiring with exactly one nonzero prime ideal is the Sierpi\'{n}ski space.
\end{theorem}
  	
\begin{proof}
Let $S$ be an entire semiring with exactly one nonzero prime ideal. This means that $\Spec(S) = \{\{0\},\mathfrak{m}\}$, with $\mathfrak{m} \neq \{0\}$. It follows that the only closed subsets of the Zariski topology of $S$ are 
  		
\begin{itemize}
\item $V(S) = \emptyset$,
\item $V(I) = V(\mathfrak{m}) = \{\mathfrak{m}\}$, for any nonzero proper ideal $I$ of $S$.
\item $V(\{0\}) = \{\{0\}, \mathfrak{m}\}$
\end{itemize}
  		
Therefore, the Zariski topology of $S$ has two points with three closed subsets which is the Sierpi\'{n}ski space. This completes the proof.
\end{proof}

\begin{remark}\label{Informationalgebrasrem}
Due to the applications of information algebras, they configure an important family of semirings. Perhaps the oldest example for information algebras is the semiring of non-negative integer numbers $\mathbb{N}_0$ equipped with usual addition and multiplication of numbers. The term ``information algebra" was introduced by Jean Kuntzmann \cite{Kuntzmann1972}. Traditionally, information algebras had some applications in graph theory \cite{Kuntzmann1972} and the theory of discrete-event dynamical systems \cite[p. 7]{Golan2003}. The other example of information algebras is the min-plus semiring $(\mathbb{R} \cup \{+\infty\}, \oplus, \otimes)$ in which its operations are defined as \[a \oplus b = \min \{a,b\} \text{~and~} a \otimes b = a+b.\] The min-plus semiring has essential applications in the shortest path problem in optimization \cite[Example 1.22]{Golan1999(b)} and is used extensively in tropical geometry \cite[\S1.1]{MaclaganSturmfels2015}. Information algebras have attracted the interests of some authors working recently on factorization problems \cite{BaethGotti2020,BaethSampson2020,ChenZhaoLiu2015}. 
\end{remark}
  
A nonempty subset $U$ of a semiring $S$ is multiplicatively closed if $(U,\cdot)$ is a submonoid of $(S,\cdot)$. The localization of a semiring $S$ at a multiplicatively closed set $U$ of $S$, denoted by $U^{-1}S$, is defined similar to its counterpart in commutative ring theory (for the details see \S5 in \cite{Nasehpour2018S}).
  
  \begin{proposition}\label{Localizationofinformationalgebras}
  	Let $U \subseteq S \setminus \{0\}$ be a multiplicatively closed subset of a semiring $S$. Then, the following statements hold:
  	
  	\begin{enumerate}
  		\item If $S$ is entire, then so is $U^{-1}S$.
  		
  		\item If $S$ is an information algebra, then so is $U^{-1}S$.
  	\end{enumerate}
  \end{proposition}

\begin{proof}
(1): This is straightforward because $a/u = 0/1$ if and only if $a = 0$, for all $a \in S$ and $u \in U$.

(2): Let $(a/u) + (b/v) = 0/1$. It follows that $va + ub = 0$. Since $S$ is zerosumfree, $va = 0$ and $ub = 0$. Now, since $u$ and $v$ are nonzero and $S$ is entire, $a$ and $b$ are both zero, and so, $a/u = 0/1$ and $b/v = 0/1$. Thus $U^{-1}S$ is an information algebra and the proof is complete.
\end{proof}

\begin{theorem}\label{Semiringidealsinformationalgebra} 
Let $S$ be a semiring. Then, the following statements hold:
  	
  	\begin{enumerate}
  		\item\label{Semiringidealszerosumfree} $\Id(S)$ is a partially ordered zerosumfree and additively idempotent semiring.
  		
  		\item $\Id(S)$ is an information algebra if and only if $S$ is entire.
  	\end{enumerate}	
  	
  \end{theorem}
  
  \begin{proof}
  	(\ref{Semiringidealszerosumfree}): By Proposition 6.29 in \cite{Golan1999(b)}, $\Id(S)$ is zerosumfree and additively idempotent. By Proposition 2.3 in \cite{Nasehpour2018P}, $(\Id(S), \subseteq)$ is a partially ordered semiring. 
  	
  	(2): Let $S$ be entire. If $I$ and $J$ are nonzero ideals, then there are nonzero elements $a \in I$ and $b \in J$. So, $ab \in IJ$ is nonzero showing that $\Id(S)$ is entire. Now, by the statement (\ref{Semiringidealszerosumfree}), $\Id(S)$ is an information algebra. Conversely, if $\Id(S)$ is an information algebra and $a$ and $b$ are nonzero elements of $S$, then the principal ideals $(a)$ and $(b)$ are nonzero. It follows that $(a)(b)$ is also a nonzero ideal of $S$. However, $(a)(b) = (ab)$ because $S$ is commutative. Therefore, $ab$ is nonzero, i.e., $S$ is entire and the proof is complete.
  \end{proof}

\section{Indigenous presemirings and their graphs}\label{sec:indigenousgraphs}
	
\begin{definition}
	A bimagma $(R,+,\cdot)$ is a (commutative) presemiring if $(R,+)$ is a commutative semigroup, $(R,\cdot)$ is a (commutative) semigroup, and $\cdot$ distributes on $+$ from both sides (see Definition 4.2.1 in \cite{GondranMinoux2008}). The presemiring $R$ is unital if there is an element $1$ in $R$ such that $r \cdot 1 = 1 \cdot r = r$, for all $r \in R$. 
\end{definition}

\begin{example}
	The bimagma $(\mathbb{N},+,\cdot)$ is a unital presemiring.
\end{example}

Inspired by Example 5 on p. 419 in \cite{Sibley2021}, we give the following definition:
	
	\begin{definition}\label{Indigenouspresemiringdef}
		Let $k$ be a positive integer. Set $I_k = \mathbb{N}_k \cup \{m\}$, where $m$ is just a symbol standing for ``many'' and is not in $\mathbb{N}_k$. Define Indigenous addition and multiplication on $I_k$ as follows: If $a, b \in \mathbb{N}_k$, then \[
	a \oplus b =
	\begin{cases}
		a+b & \text{if $ a+b \leq k$} \\
		m & \text{if $a+b > k$}
	\end{cases}
	\text{, and~} a \odot b =
	\begin{cases}
		ab & \text{if $ab\leq k$} \\
		m & \text{if $ab > k$}
	\end{cases}
	\] and if either $a= m$ or $b = m$, then \[a \oplus b = a \odot b = m.\]
	\end{definition}
	
	\begin{proposition}\label{Indigenoussemiringpro}
	 The bimagma $(I_k, \oplus, \odot)$ defined in Definition \ref{Indigenouspresemiringdef} is a unital presemiring and an epimorphic image of the presemiring $\mathbb{N}$.
	\end{proposition}

\begin{proof}
It is easy to see that $(I_k, \oplus)$ is a commutative semigroup and $(I_k,\odot,1)$ is a commutative monoid. Now, let $a$, $b$, and $c$ be elements of $I_k$. If $ab + ac$ is less than $k+1$, then $ab$ and $ac$ are also less than $k+1$, and so, the distributive laws hold because the computation is done in natural numbers. If at least one of the elements $ab$ and $ac$ are greater than $k$, then their addition is $m$ and we have \[a \odot (b \oplus c) = m = (a \odot b) \oplus (a \odot c).\] This means that $(I_k, \oplus, \odot)$ is a unital presemiring. It is easy to check that the function $f: \mathbb{N} \rightarrow I_k$ defined by \[f(x) = \begin{cases}
	x & \text{if $x \leq k$} \\
	m & \text{if $x>k$}  
\end{cases}.\] is a presemiring epimorphism and the proof is complete.
\end{proof}

\begin{examples}\label{historicalexamplesIndigenous} Let $k$ be a positive integer number. We call the presemiring $I_k$ an Indigenous presemiring of order $k$. In the following, we give some examples mainly discussed in the literature.
\begin{enumerate}
	\item Gordon illustrated that the Pirah\~{a} applied a numerical vocabulary corresponding to the terms ``h\'{o}i'' (for ``one''), ``ho\'{i}'' (for ``two''), and ``baagiso'' (for ``many'') \cite{Gordon2004}. One may correspond this to the Indigenous presemiring of order $2$.  
	
	\item (Sibley's Indigenous presemiring) As we have already explained, we were inspired by an example given on p. 419 in Sibley's book \cite{Sibley2021}. Sibley's Indigenous presemiring is the Indigenous presemiring of order $3$.
	
	\item On p. 5 in his book \cite{Ifrah2000} on the universal history of numbers, Ifrah explains that the Botocudos had only two real terms for numbers: one for ``one'', and the other for ``a pair''. With these lexical items they could manage to express three and four by saying something like ``one and two'' and ``two and two''. However, these people had as much difficulty conceptualizing a number above four. In fact, for larger numbers, some of the Botocudos just pointed to their hair as if they were trying to say there are as ``many'' as there are hairs on their head. This may correspond to the Indigenous presemiring of order $4$.
	
	\item In his academic book \cite{Sommerfelt1938}, Sommerfelt reports that the Aranda had only two number terms, ``ninta'' for one, and ``tara'' for two. Three and four were expressed as ``tara-mi-ninta'' (i.e., two and one) and ``tara-ma-tara'' (i.e., two and two), respectively, and the number series of Aranda stopped there. For larger quantities, imprecise terms resembling ``a lot'', ``several'', and so on were used. This may also correspond to the Indigenous presemiring of order $4$.
\end{enumerate}	
\end{examples}

\begin{definition}\label{Indigenousgraphs} 
We define the Indigenous graph $\IG_k$ as follows:
\begin{enumerate}
	\item The set of the vertices of $\IG_k$ is $I_k = \mathbb{N}_k \cup \{m\}$.
	
	\item The doubleton $\{a,b\}$ is an edge of $\IG_k$ if $a \neq b$ and their multiplication $a \odot b$ equals $m$. 
\end{enumerate} 
\end{definition}

Let us recall that the diameter of a graph $G$, denoted by $\diam(G)$, is the greatest distance between two vertices of $G$ \cite[\S3.1.7]{BondyMurty2008}.

\begin{theorem}\label{diamIG}
For any positive integer $k$, the Indigenous graph $\IG_k$ is a connected graph with $\diam(\IG_1) = 1$ and $\diam(\IG_k) = 2$ if $k > 1$.
\end{theorem}

\begin{proof}
The Indigenous graph $\IG_1$ has only two vertices $1$ and $m$ and they are connected because $1 \odot m = m$. Therefore, $\IG_1$ is the complete graph $K_2$ which is a connected graph with $\diam(\IG_1) = 1$. 

Now, let $k \geq 2$ and $a \neq b$ be arbitrary elements in $\mathbb{N}_k$. Since $a \odot m = b \odot m = m$, $a$ is connected to $m$ and $m$ is connected to $b$. Therefore, $\IG_k$ is a connected graph with $\diam(\IG_k) \leq 2$. However, $1$ and $k$ are not connected. So, $\diam(\IG_k) > 1$. This completes the proof.
\end{proof}

Let us recall that if a graph $G$ has at least one cycle, the length of a shortest cycle is its girth \cite[p. 42]{BondyMurty2008}. If a graph has no cycle, its girth is defined to be infinity. The girth of a graph is usually denoted by $g(G)$.

\begin{theorem}\label{girthIG}
The girth of the Indigenous graph $\IG_k$ is $3$ if $k \geq 3$, and is infinity, otherwise.
\end{theorem}

\begin{proof}
Let $k \geq 3$. Then, the vertices $k$ and $k-1$ are connected because $k(k-1) > k$ in $\mathbb{N}$, and so, $k \odot (k-1) = m$ in $I_k$. Therefore, the triangle with the vertices $k-1$, $k$, and $m$ is a subgraph of $\IG_k$, and so, $g(\IG_k) = 3$. It is evident that the graph $\IG_1$ has no cycle. Also, in the graph $\IG_2$, the vertices $1$ and $2$ are not connected and again $\IG_2$ has no cycle. This completes the proof.
\end{proof}

Recall that a clique of a graph is a set of mutually adjacent vertices, and that the maximum size of a clique of a graph $G$, the clique number of $G$, is denoted by $\omega(G)$ \cite[p. 296]{BondyMurty2008}.

\begin{proposition}\label{cliqueIGleq4}
The clique number of $\IG_1$, $\IG_2$, $\IG_3$, and $\IG_4$ is $2$, $2$, $3$, and $4$,  respectively.
\end{proposition}

\begin{proof} We compute the clique number of $\IG_k$ for $k \leq 4$ as follows:
	
\begin{enumerate}
	\item In $\IG_1$, the vertices $1$ and $m$ are adjacent and its clique number of $\IG_1$ is $2$.
	
	\item In $\IG_2$, the vertex $m$ is adjacent to the vertices $1$ and $2$ but the vertices $1$ and $2$ are not adjacent. So, the clique number of $\IG_2$ is again $2$.
	
	\item In $\IG_3$, the vertices $m$, $2$, and $3$ are mutually adjacent while $1$ is not connected to $2$ and $3$. It follows that the clique number of $\IG_3$ is $3$.
	
	\item In $\IG_4$, the vertices $m$, $2$, $3$, and $4$ are mutually adjacent while $1$ is not connected to $2$, $3$, and $4$. This means that the clique number of $\IG_4$ is $4$.
\end{enumerate}

This completes the proof.
\end{proof}

\begin{theorem}\label{cliqueIG}
The clique number of the Indigenous graph $\IG_k$ is at least $\left\lfloor \frac{k}{2} \right\rfloor + 1$, for any positive integer $k$.
\end{theorem}

\begin{proof}
In view of Proposition \ref{cliqueIGleq4}, the inequality $\omega(\IG_k) \geq \left\lfloor \frac{k}{2} \right\rfloor + 1$ holds for all $k \leq 4$. Now, let $k \geq 5$. We distinguish two cases.

Case I. If $k = 2 \alpha$ is an even number with $\alpha \geq 3$, then \[\left(k - \left\lfloor \frac{k}{2} \right\rfloor \right)^2 - k = \alpha^2 - 2\alpha > 0.\]

Case II. If $k = 2\alpha + 1$ is an odd number with $\alpha \geq 2$, then \[\left(k - \left\lfloor \frac{k}{2} \right\rfloor \right)^2 - k = (\alpha+1)^2 - (2\alpha + 1) = \alpha^2 > 0.\] Therefore, all vertices \[\left(k - \left\lfloor \frac{k}{2} \right\rfloor\right), \left(k - \left\lfloor \frac{k}{2} \right\rfloor\right) + 1, \dots, k, m\] are mutually adjacent to each other in $\IG_k$. Thus $\omega(\IG_k) \geq \left\lfloor \frac{k}{2} \right\rfloor + 1$ for any positive integer number $k$ and the proof is complete.
\end{proof}

Recall that the smallest number of colors needed to color the vertices of a graph $G$ such that no two adjacent vertices share the same color is called the chromatic number of $G$, denoted by $\chi(G)$ \cite[p. 358]{BondyMurty2008}. It is clear that $\chi(G) \geq \omega(G)$ \cite[p. 359]{BondyMurty2008}. So, we have the following corollary:

\begin{corollary}\label{chromaticIG}
The chromatic number of the Indigenous graph $\IG_k$ is at least $\left\lfloor \frac{k}{2} \right\rfloor + 1$, for any positive integer $k$.
\end{corollary}

\section{Indigenous semirings and their ideals}\label{sec:indigenousideals}

\begin{proposition}\label{presemiringembedhemiring}
	Any presemiring can be embedded into a hemiring.	
\end{proposition}

\begin{proof}
	Let $E$ be a presemiring and suppose that $0 \notin E$. Set $E^{\prime} = E \cup \{0\}$ and define $+^{\prime}$ and ${\cdot}^{\prime}$ on $E^{\prime}$ as follows:
	
	\begin{itemize}
		
		\item $a +^{\prime} b = a + b$ for all $a,b\in E$ and $a +^{\prime} 0 = 0 +^{\prime} a = a$ for all $a\in E^{\prime}$.
		\item $a {\cdot}^{\prime} b = a \cdot b$ for all $a,b\in E$ and $a {\cdot}^{\prime} 0 = 0 {\cdot}^{\prime} a = 0$ for all $a\in E^{\prime}$.
		
	\end{itemize}
	
	One can easily check that $(E^{\prime}, +^{\prime}, {\cdot}^{\prime})$ is also a presemiring and the element $0$ is an identity element for addition and an absorbing element for multiplication. Now, define $\varphi : E \rightarrow E^{\prime}$ by $\varphi(x) = x$. It is clear that $\varphi$ is a presemiring monomorphism. This completes the proof.
\end{proof}

\begin{corollary}\label{Indigenoussemiringcor}
The Indigenous presemiring $I_k$ can be embedded into the semiring $I_k \cup \{0\}$, for each $k \in \mathbb{N}$. 
\end{corollary}

\begin{definition}\label{Indigenoussemiringdef}
For any positive integer $k$, we call the semiring $I_k \cup \{0\}$ given in Corollary \ref{Indigenoussemiringcor}, the Indigenous semiring and denote it by $S_k$.
\end{definition}

\begin{theorem}\label{GeneralpropertiesofIndigenoussemirings}
Let $k$ be a positive integer and $S_k$ the Indigenous semiring. Then, the following statements hold:

\begin{enumerate}
	\item\label{IndigenoustotallyorderedIA} The semiring $S_k$ is a totally ordered information algebra with the smallest element $0$ and the largest element $m$. However, $S_k$ is not a semidomain.
	
	\item\label{LocalizationIndigenous} Let $k > 1$ and $U$ be a multiplicatively closed set in $S_k$ having a positive integer $a > 1$ and $0 \notin U$. Then, the localization $U^{-1} S_k$ of $S_k$ at $U$ is isomorphic to the Boolean semiring $\mathbb{B} = \{0,1\}$.
	
	\item\label{Indigenouslocal} $(S_k, S_k \setminus\{1\})$ is a local semiring.
	
	\item\label{misinanynonzeroideal} The set $\mathfrak{s}_k = \{0,m\}$ is the smallest nonzero ideal of the Indigenous semiring $S_k$. In particular, any nonzero ideal of $S_k$ possesses $m$.
	
	\item\label{Subtractiveidealsindigenoussemirings} The semiring $S_k$ is austere, i.e., the only subtractive ideals of $S_k$ are $\{0\}$ and $S_k$ (see p. 71 in \cite{Golan1999(b)}).
	
	\item\label{PrimeidealsIndigenoussemiring} The only prime ideals of the Indigenous semiring $S_k$ are $\{0\}$ and $S_k \setminus \{1\}$.
	
	\item\label{Nonzeroprincipalprime} $S_k$ has a nonzero principal prime ideal if and only if $k\leq 2$.
	
	\item\label{ZariskitopologyIndigenoussemiring} The Zariski topology of the Indigenous semiring $S_k$ is the Sierpi\'{n}ski space.
	
	\item\label{RadicalidealsIndigenoussemirings} If $I$ is a nonzero proper ideal of $S_k$, then $\sqrt{I} = S_k \setminus \{1\}$. In particular, the only radical ideals of $S_k$ are $\{0\}$, $S_k \setminus \{1\}$, and $S_k$.
	
	\item\label{SemiringidealsIndigenousinformationalgebra} $\Id(S_k)$ is a partially ordered information algebra and if $I$ is a proper nonzero ideal of $S_k$, then \[\{0\} \subseteq \mathfrak{s}_k \subseteq I \subseteq \mathfrak{m}_k \subseteq S_k.\]
	
	\item\label{Absorbingelementnonzeroideals} $(\Id(S_k)\setminus\{\mathbf{0}\}, \cdot)$ is a monoid with the absorbing element $\mathfrak{s}_k$, where by $\mathbf{0}$, we mean the zero ideal $\{0\}$ of $S_k$. 
	
	\item\label{Monoidnonzeroidealsnilpotent} If $n$ is a positive integer number such that $2^n > k$, then for any nonzero proper ideals $\{I_i\}_{i=1}^{n}$ of $S_k$, we have $\prod_{i=1}^{n} I_i= \mathfrak{s}_k$; in other words, the multiplicative monoid $M = \Id(S_k)\setminus\{\mathbf{0}\}$ is nilpotent (i.e., there is a positive integer number $n$ with $M^n = \{\mathfrak{s}_k\}$). 
\end{enumerate}
\end{theorem}

\begin{proof}
(\ref{IndigenoustotallyorderedIA}): If $a$ and $b$ are nonzero elements of $S_k$, then their multiplication (addition) is either a positive integer number less than $k+1$ or $m$. So, $S_k$ is entire and zerosumfree. If we set \[0 < 1 < \cdots < k < m,\] then it is easy to see that $a \leq b$ implies $a+c \leq b+c$ and $ac \leq bc$, for all $a$, $b$, and $c$ in $S_k$. This means that $S_k$ is a totally ordered information algebra with the smallest element $0$ and the largest element $m$. Note that while $m \neq 1$, we always have \[m \cdot 1 = m = m \cdot m\] which means that $S_k$ is not a semidomain.

(\ref{LocalizationIndigenous}): Since $a \in U$ and $U$ is multiplicatively closed, $a^n \in U$ for each natural number $n$. It is clear that for sufficiently large enough $n$, $a^n > k$ in $\mathbb{N}$, and so, $a^n$ which is $m$ in $S_k$ is an element of $U$. Now, let $a / u$ be a nonzero element in $U^{-1} S_k$. Note that since $m \in U$ and $m / m = 1/1$ is the multiplicative identity of the semiring $U^{-1} S_k$, we have \[a/u = (a/u)(m/m) = (am)/(mm) = m/m\] showing that $U^{-1} S_k$ has only two elements $0/1$ and $1/1$. On the other hand, \[(1/1) + (1/1) = (m/m) + (m/m) = (m+m)/m = m/m = 1/1.\] This shows that $U^{-1} S_k$ is isomorphic to the Boolean semiring $\mathbb{B}$.

(\ref{Indigenouslocal}): Let $a \neq 1$ and $b \neq 1$. It is easy to see that $a+b \neq 1$. Also, since $xy = 1$ if and only if $x = 1$ and $y = 1$ in $S_k$, we see that if $s \in S_k$ is arbitrary and $a \neq 1$, then $sa \neq 1$. This means that $\mathfrak{m}_k = S_k \setminus \{1\}$ is an ideal of $S_k$. However, there is no other ideals strictly between $\mathfrak{m}_k$ and $S_k$. So, $\mathfrak{m}_k$ is a maximal ideal of $S_k$. On the other hand, an ideal $I$ of a semiring is proper if and only if $1 \notin I$. Therefore, any proper ideal of $S_k$ is a subset of $\mathfrak{m}_k$. Thus $\mathfrak{m}_k$ is the only maximal ideal of $S_k$.

(\ref{misinanynonzeroideal}): It is easy to see that $\mathfrak{s}_k = \{0,m\}$ is an ideal of $S_k$. Now, let $I$ be a nonzero ideal of $S_k$. If $s$ is a nonzero element of $I$, then $m = ms$ must be in $I$.

(\ref{Subtractiveidealsindigenoussemirings}): Let $I$ be a nonzero proper ideal of $S_k$. Then by the statement (\ref{misinanynonzeroideal}), $m \in I$ while $1 \notin I$. However, $m+1 = m \in I$. This means that $I$ is not subtractive. It is evident that $\{0\}$ and $S_k$ are subtractive.

(\ref{PrimeidealsIndigenoussemiring}): By the statement (\ref{IndigenoustotallyorderedIA}), $\{0\}$ in prime. In view of the statement (\ref{Indigenouslocal}) and Corollary 7.13 in \cite{Golan1999(b)}, $S_k \setminus \{1\}$ is also prime. Now, let $P$ be a nonzero prime ideal of $S_k$. If $a \in S_k \setminus \{0,1\}$, then either $2 \leq a \leq k$ or $a = m$. In any case, there is a positive integer $n$ such that $a^n = m$. On the other hand, by the statement (\ref{misinanynonzeroideal}), $m$ is an element of each nonzero ideal of $S_k$. Consequently, $a^n \in P$. Since $P$ is prime, we have $a \in P$. Therefore, $P = S_k \setminus \{1\}$.

(\ref{Nonzeroprincipalprime}): In view of the statement (\ref{PrimeidealsIndigenoussemiring}), in $S_1$, the principal ideal $(m) = S_1 \setminus \{0\}$ is prime, and in $S_2$, $(2) = S_2 \setminus \{1\}$ is also prime. Now, let $k \geq 3$. The principal ideal $(2)$ is not prime because a suitable power of $3$ is $m \in (2)$ but $3$ is not an element of $(2)$. Now, let $p > 2$ be a prime number in $\mathbb{N}_k$. The principal ideal $(p)$ is not prime because a suitable power of $2$ is $m \in (p)$ while $2$ is not in $(p)$. If $c$ is a composite number in $\mathbb{N}_k$, then the principal ideal $(c)$ is clearly not prime. Also, $(m) = \{0,m\}$ is not prime because a suitable power of $2$ is $m \in (m)$ while $2$ is not in $(m)$.

(\ref{ZariskitopologyIndigenoussemiring}): By the statement (\ref{PrimeidealsIndigenoussemiring}), the only prime ideals of $S_k$ are $\{0\}$ and $\mathfrak{m}_k = S_k \setminus \{1\}$. Therefore, by Theorem \ref{Zariskitopologyentiresemiringwithtwoprimes}, the Zariski topology of $S_k$ is the Sierpi\'{n}ski space.

(\ref{RadicalidealsIndigenoussemirings}): For any ideal $I$ of $S_k$, the radical $\sqrt{I}$ of $I$ is the intersection of prime ideals of $S_k$ containing $I$ (cf. Theorem 3.2 in \cite{Nasehpour2018P}). Now, if $I$ is nonzero, then $\sqrt{I} = S_k \setminus \{1\}$ because by the statement (\ref{PrimeidealsIndigenoussemiring}), $S_k \setminus \{1\}$ is the only prime ideal of $S_k$ containing $I$. Therefore, a nonzero proper ideal of $S_k$ is radical if and only if $I = S_k \setminus \{1\}$. Now, since by the statement (\ref{IndigenoustotallyorderedIA}), $S_k$ is entire, $\{0\}$ is a radical ideal. It is evident that $S_k$ is a radical ideal.

(\ref{SemiringidealsIndigenousinformationalgebra}): By Theorem \ref{Semiringidealsinformationalgebra}, $\Id(S_k)$ is a partially ordered information algebra. By the statement (\ref{misinanynonzeroideal}), $\mathfrak{s}_k$ is the smallest nonzero ideal. By the statement (\ref{Indigenouslocal}), $\mathfrak{m}_k = S_k \setminus \{1\}$ is the largest proper ideal.

(\ref{Absorbingelementnonzeroideals}): Since $\Id(S_k)$ is an entire semiring, $(\Id(S_k)\setminus\{\mathbf{0}\}, \cdot)$ is a monoid. Now, let $I$ be a nonzero ideal. We need to prove that $I \cdot \mathfrak{s}_k = \mathfrak{s}_k$. By the statement (\ref{misinanynonzeroideal}), $\mathfrak{s}_k$ is the smallest nonzero ideal of $S_k$. On the other hand, \[I \cdot \mathfrak{s}_k \subseteq I \cap \mathfrak{s}_k = \mathfrak{s}_k.\]

(\ref{Monoidnonzeroidealsnilpotent}): Let $n$ be a positive integer number with $2^n > k$. Let $\{I_i\}_{i=1}^{n}$ be arbitrary nonzero proper ideals of $S_k$. By the statement (\ref{misinanynonzeroideal}), $m \in I_i$ while $1 \notin I_i$. Consider $a_i \in I_i$. Observe that if at least one of the $a_i \in I_i$ is zero, then $\prod_{i=1}^{n} a_i = 0$. Also, if all of the $a_i$s are nonzero and at least one of them is $m$, then $\prod_{i=1}^{n} a_i = m$. Now, let $a_i \notin \{0,m\}$. This means that $2 \leq a_i$, for each $i$, and so, in $\mathbb{N}$, we have \[\prod_{i=1}^{n} a_i \geq 2^n > k.\] This means that $\prod_{i=1}^{n} a_i = m$ in $S_k$ and the proof is complete.
\end{proof}

\begin{theorem}
Let $S_k$ be the Indigenous semiring and $M$ a commutative monoid. Then, the following statements hold:

\begin{enumerate}
	\item The monoid semiring $S_k[M]$ is an information algebra.
	
	\item If $M$ is a totally ordered commutative monoid, then the function \[\deg: (S_k[M],+,\cdot,0,1) \rightarrow (M_{\infty},\max,+,-\infty,0)\] is a semiring morphism. 
\end{enumerate}
\end{theorem}

\begin{proof}
(1): By Theorem 4.10 in \cite{Nasehpour2025}, if $M$ is a commutative monoid and $S$ an information algebra, then the monoid semiring $S[M]$ is also an information algebra. Thus in view of Theorem \ref{GeneralpropertiesofIndigenoussemirings}, $S_k[M]$ is an information algebra.

(2): In view of Theorem \ref{GeneralpropertiesofIndigenoussemirings}, this is a special case of Corollary 4.18 in \cite{Nasehpour2025}. This completes the proof. 
\end{proof}

\section{Distinguished elements of polynomials and formal power series over the Indigenous semirings}\label{sec:distinguishedelements}

In this section, $S_k$ denotes the Indigenous semiring defined in Definition \ref{Indigenoussemiringdef}. Since $S_k$ is an entire semiring, $S_k[X]$ \cite[Corollary 2.4]{Nasehpour2021} and $S_k[[X]]$ \cite[Lemma 43]{Nasehpour2016} are also entire semirings. Therefore, $S_k$, $S_k[X]$, and $S_k[[X]]$ have no nontrivial zero-divisors (and nilpotent elements). Now, we proceed to discuss their units and idempotent elements.

\begin{proposition}\label{UnitsIndigenoussemirings}
The only unit element of $S_k$, $S_k[X]$, and $S_k[[X]]$ is $1$.
\end{proposition}

\begin{proof}
	Obviously in $S_k$, $ab = 1$, implies that $a = b = 1$. So, the only unit element of $S_k$ is $1$. Let $f,g \in S_k[X]$ with $fg = 1$. Since $S_k$ is an entire semiring, we obtain that $\deg(f) = \deg(g) = 0$. Therefore, $f = g = 1$.
	
	Now, let $f, g \in S_k[[X]]$ with $fg = 1$. Suppose that $f = \sum_{i=0}^{+\infty} a_i X^i$ and $g = \sum_{j=0}^{+\infty} b_j X^j$.  Clearly, $a_0 = b_0 = 1$. Since $S_k$ is zerosumfree, from $fg = 1$, we obtain that $a_i = b_i = 0$, for all $i \geq 1$. This completes the proof.  
\end{proof}

\begin{proposition}\label{IdempotentsIndigenoussemirings1}
The only idempotent elements of $S_k$ and $S_k[X]$ are $0$, $1$, and $m$.
\end{proposition}

\begin{proof}
In $S_k$, if $a$ is different from $0$, $1$, and $m$, then $2 \leq a \leq k$. If $a^2 \leq k$, then $a^2 \neq a$. If $a^2 > k$, then $a^2 = m$ which means that again $a^2 \neq a$. Thus the only idempotent elements of $S_k$ are $0$, $1$, and $m$.

Now, let $f \in S_k[X]$. Since $S_k$ is an entire semiring, then \[\deg(f^2) > \deg(f)\] except the case that $f$ is a constant polynomial. Therefore, $f^2 = f$ if and only if $f = a$, where $a \in S_k$. This means that $f$ is idempotent in $S_k[X]$ if and only if $f$ is either $0$ or $1$ or $m$. This completes the proof.
\end{proof}

\begin{theorem}\label{IdempotentsIndigenoussemirings2}
An element $f$ in $S_k[[X]]$ is idempotent if and only if either $f = 0$ or $f = 1$ or $f = m$ or \[f = a_0 + \sum_{i=1}^{+\infty} m X^{s_i},\] where $a_0 = 1,m$ and the set $\{s_i\}_{i=1}^{+\infty}$ is a subsemigroup of $(\mathbb{N},+)$.
\end{theorem}

\begin{proof}
Let $f = \sum_{i=0}^{+\infty} a_i X^i$ be idempotent. It follows that $a^2_0 = a_0$, and so, by using Proposition \ref{IdempotentsIndigenoussemirings1}, we can distinguish three cases:

\begin{enumerate}
	\item The case $a_0 = 0$. If $a_0 = 0$, then $f = \sum_{i=1}^{+\infty} a_i X^i$. Now, if $a_i \neq 0$, for some $i > 0$, then obviously $f^2 \neq f$. Therefore, $f = 0$.
	\item The case $a_0 = 1$. Our claim is that $a_i \in \{0,m\}$, for any $i > 0$. This is because if $a_i \neq 0,m$, for some $i > 0$, then the coefficient of $X^i$ in $f^2$ is at least $2a_i$ in $\mathbb{N}$ which is never $a_i$ in $S_k$. Now, consider \[f = 1 + mX^{s_1} + mX^{s_2} + \cdots.\] Our claim is that $f$ is idempotent if and only if the set $E = \{s_i\}_{i=1}^{+\infty}$ is a subsemigroup of $(\mathbb{N},+)$. For the direct implication, we need to show that $s_i + s_j \in E$, for all $s_i \in E$ and $s_j \in E$. Consider the monomials $mX^{s_i}$ and $mX^{s_j}$ in $f$. Therefore, $mX^{s_i + s_j}$ is a monomial in $f^2$. Since $f^2 = f$, $s_i + s_j$ needs to be one of the exponents of the monomials of $f$ which means that $E$ is a subsemigroup of $\mathbb{N}$. For the converse implication, observe that an easy calculation shows that if $\{s_i\}_{i=1}^{+\infty}$ is a subsemigroup of $\mathbb{N}$, then $f^2 = f$. 
	
	\item The case $a_0 = m$. Our claim is that in this case also $a_i \in \{0,m\}$, for any $i > 0$. This is because if $a_i \neq 0,m$, for some $i > 0$, then the coefficient of $X^i$ in $f^2$ is \[a_0 a_i + \dots + a_ia_0 = m a_i + \dots + a_i m = m\] which is never the same as $a_i$, i.e., the coefficient $X^i$ in $f$. Now, consider \[f = m + mX^{s_1} + mX^{s_2} + \cdots.\] Similar to the second case, one can easily check that $f$ is idempotent if and only if the set $\{s_i\}_{i=1}^{+\infty}$ is a subsemigroup of $(\mathbb{N},+)$.
\end{enumerate}
This completes the proof.
\end{proof}

\begin{theorem}
Let $\alpha \neq 0$ and $\beta$ be elements of the Indigenous semiring $S_k$ and $X$ an indeterminate over $S_k$. Then, $f=\alpha X^2+\beta$ is irreducible if and only if one of the following cases happens:

\begin{enumerate}
	\item $\alpha$ and $\beta$ are in $\mathbb{N}_k$ and $\gcd(\alpha, \beta) = 1$.
	
	\item $\alpha = m$ and $\beta = 1$.
	
	\item $\alpha = 1$ and $\beta = m$.
\end{enumerate}

\end{theorem}

\begin{proof}
In view of Proposition 6.2 in \cite{Nasehpour2025} and Theorem \ref{GeneralpropertiesofIndigenoussemirings} in the current paper, $f$ cannot be factored into $g = aX+b$ and $h = cX+d$ in $S_k[X]$, where $a$ and $c$ are nonzero in $S_k$. Therefore, $f$ is reducible if and only if there is a nonzero nonunit $\gamma$ in $S_k$, i.e., $\gamma \neq 0,1$ such that $f = \gamma g$, for some $g = \alpha'X^2 + \beta' \in S_k[X]$ which means that \[\alpha = \gamma \alpha' \wedge \beta = \gamma \beta'.\] Note that if $\alpha = \beta = m$, then $f$ is reducible. Therefore, if $f$ is irreducible, then at least one of the coefficients of $f$ must be a natural number. Now, observe the following: 

\begin{enumerate}
	\item If $\alpha$ and $\beta$ are natural numbers, then $\gamma \neq m$, and $f$ is reducible if and only if $ \gamma \mid \gcd(\alpha,\beta)$, i.e., $\gcd(\alpha,\beta) > 1$.
	
	\item If $\alpha = m$, then $f$ is reducible if and only if $\beta > 1$ because $mX^2 + 1$ is irreducible and \[f = mX^2 + \beta = \beta (mX^2 + 1).\]
	
	\item If $\beta = m$, then $f$ is reducible if and only if $\alpha > 1$ because $X^2 + m$ is irreducible and \[f = \alpha X^2 + m = \alpha (X^2 + m).\] 
\end{enumerate}
This gives the characterization of all irreducible polynomials of the form $\alpha X^2+\beta$ and the proof is complete.
\end{proof}

\subsection*{Acknowledgments} The authors wish to thank the anonymous referees for their valuable comments and suggestions, which improved the presentation of this paper.

\bibliographystyle{plain}

\begin{thebibliography}{15.}
	
\bibitem{ArensDugundji1951} Arens, R., Dugundji, J.: Topologies for function spaces. Pac. J. Math. 1, 5--31 (1951).

\bibitem{BaethGotti2020} Baeth, N.R., Gotti, F.: Factorizations in upper triangular matrices over information semialgebras. J. Algebra 562, 466--496 (2020).

\bibitem{BaethSampson2020} Baeth, N.R., Sampson, R.: Upper triangular matrices over information algebras. Linear Algebra Appl. 587, 334--357 (2020).

\bibitem{BenderBeller2018} Bender, A., Beller, S.: Numeration systems as cultural tools for numerical cognition. In Language and culture in mathematical cognition, 297--320. Academic Press, (2018).

\bibitem{BenderBeller2021} Bender, A., Beller, S.: Ways of counting in Micronesia. Hist. Math. 56, 40--72 (2021).


\bibitem{BondyMurty2008} Bondy, J.A., Murty, U.S.R.: Graph theory. Graduate Texts in Mathematics 244. Berlin: Springer. xii, 651 p. (2008).
	
\bibitem{Bourne1951} Bourne, S.: The Jacobson radical of a semiring. Proc. Natl. Acad. Sci. USA 37, 163--170 (1951).

\bibitem{ChenZhaoLiu2015} Chen, Y., Zhao, X.; Liu, Z.: On upper triangular nonnegative matrices. Czech. Math. J. 65, No. 1, 1--20 (2015).
	
\bibitem{Everett2005} Everett, D.L.: Cultural constraints on grammar and cognition in Pirah\~{a}. Current Anthropology, 46, 621--646 (2005).

\bibitem{Gillilland1975} Gillilland, C.L.C.: Stone money of Yap: a numismatic survey. Washington: Smithsonian Institution Press. Smithsonian Studies in History and Technology. 23, 75 p. (1975).

\bibitem{Golan2003} Golan, J.S.: Semirings and affine equations over them: theory and applications. Mathematics and its Applications (Dordrecht) 556. Dordrecht: Kluwer Academic Publishers. xiv, 241 p. (2003).
	
\bibitem{Golan1999(b)} Golan, J.S.: Semirings and their applications. Dordrecht: Kluwer Academic Publishers. xi, 381 p. (1999).
		
\bibitem{GondranMinoux2008} Gondran, M., Minoux, M.: Graphs, dioids and semirings. New models and algorithms. Operations Research/Computer Science Interfaces Series 41. Dordrecht: Springer. xix, 383 p. (2008).

\bibitem{Gordon2004} Gordon, P.: Numerical cognition without words: Evidence from Amazonia. Science, 306, 496--499 (2004).

\bibitem{Ifrah2000} Ifrah, G.: The universal history of numbers. From prehistory to the invention of the computer. Translated from the 1994 French original by David Bellos, E.F. Harding, Sophie Wood and Ian Monk. New York, NY: Wiley. xxii, 633 p. (2000).

\bibitem{Kuntzmann1972} Kuntzmann J.: Th\'{e}orie des r\'{e}seaux: graphes. Paris: Dunod. 228 p. (1972).

\bibitem{MaclaganSturmfels2015} Maclagan, D., Sturmfels, B.: Introduction to tropical geometry. Graduate Studies in Mathematics 161. Providence, RI: American Mathematical Society (AMS). xii, 363 p. (2015).

\bibitem{Nasehpour2025} Nasehpour, P.: Distinguished elements in semiring extensions. J. Algebra Appl. 24, No. 3, Article ID 2550071, 26 p. (2025).

\bibitem{Nasehpour2016} Nasehpour, P.: On the content of polynomials over semirings and its applications. J. Algebra Appl. 15, No. 5, Article ID 1650088, 32 p. (2016).

\bibitem{Nasehpour2021} Nasehpour, P.: On zero-divisors of semimodules and semialgebras. Georgian Math. J. 28, No. 3, 413--428 (2021).

\bibitem{Nasehpour2018P} Nasehpour, P.: Pseudocomplementation and minimal prime ideals in semirings. Algebra Univers. 79, No. 1, Paper No. 11, 15 p. (2018).

\bibitem{Nasehpour2018S} Nasehpour, P.: Some remarks on ideals of commutative semirings. Quasigroups Relat. Syst. 26, No. 2, 281--298 (2018).

\bibitem{Nasehpour2019} Nasehpour, P.: Some remarks on semirings and their ideals. Asian-Eur. J. Math. 12, No. 7, Article ID 2050002, 14 p. (2019).

\bibitem{Rosenfeld1968} Rosenfeld, A.: An introduction to algebraic structures. San Francisco-Cambridge-London-Amsterdam: Holden-Day. XIV, 285 p. (1968).

\bibitem{Rotman1988} Rotman, J.J.: An introduction to algebraic topology. Graduate Texts in Mathematics, 119. New York (USA) etc.: Springer-Verlag. xiii, 433 p. (1988).
		
\bibitem{Sibley2021} Sibley, T.Q.: Thinking algebraically. An introduction to abstract algebra. AMS/MAA Textbooks 65. Providence, RI: American Mathematical Society (AMS). xiii, 478 p. (2021).

\bibitem{Sommerfelt1938} Sommerfelt, A.: La langue et la soci\'{e}t\'{e}. Oslo: H. Aschehoug \& Co., 233 p. (1938).

\bibitem{Vandendriessche2022} Vandendriessche, E.: The concrete numbers of ``primitive'' societies: a historiographical approach. Hist. Math. 59, 12--34 (2022). 	
\end{thebibliography}

\end{document}